\documentclass[11pt]{article} 
\usepackage[utf8]{inputenc}

\usepackage{geometry}
\usepackage{paralist}

\usepackage{amsthm,amsmath,amssymb}
\usepackage{mathtools}

\usepackage{subcaption}

\usepackage{xcolor}
\usepackage{hyperref}
\hypersetup{
	colorlinks=false,
	linkbordercolor=gray,
	citebordercolor=gray,
	urlbordercolor=gray}%
\usepackage{cleveref}

\usepackage{sectsty}
\allsectionsfont{\sffamily\mdseries\scshape\centering\nohang}

\usepackage{abstract}

\usepackage{caption}
\usepackage{sidecap}

\allowdisplaybreaks

\newtheoremstyle{mplain}
{\topsep}   
{\topsep}   
{\itshape}  
{0pt}       
{\scshape} 
{:}         
{5pt plus 1pt minus 1pt} 
{}          

\newenvironment{scproof}[1][\proofname]{\proof[\textsc{#1}]}{\endproof}

\theoremstyle{mplain}

\newtheorem{theorem}{Theorem}[]
\crefname{theorem}{Theorem}{Theorems} 
\Crefname{theorem}{Theorem}{Theorems}

\newtheorem{lemma}[theorem]{Lemma}
\crefname{lemma}{Lemma}{Lemmas} 
\Crefname{lemma}{Lemma}{Lemmas}

\newtheorem{proposition}[theorem]{Proposition}
\crefname{proposition}{Proposition}{Propositions} 
\Crefname{proposition}{Proposition}{Propositions}

\crefname{cor}{Corollary}{Corollaries}

\newtheorem{definition}[theorem]{Definition}
\crefname{definition}{Definition}{Definitions} 
\Crefname{definition}{Definition}{Definitions}

\newtheorem{remark}[theorem]{Remark}
\crefname{remark}{Remark}{Remarks} 
\Crefname{remark}{Remark}{Remarks}

\newtheorem{conj}[theorem]{Conjecture}
\Crefname{conj}{Conjecture}{Conjectures}

\newtheorem{problem}[theorem]{Problem}
\crefname{problem}{problem}{problems} 
\Crefname{problem}{Problem}{Problems}

\newtheorem{claim}{Claim}[]
\crefname{claim}{Claim}{Claims} 
\Crefname{claim}{Claim}{Claims}

\numberwithin{theorem}{section}

\newcommand{\E}{\mathbb{E}}
\renewcommand{\Pr}{\mathbb{P}}
\newcommand{\eps}{\varepsilon}
\newcommand{\Op}{O_\mathbb{P}}

\newcommand{\F}{\mathcal{F}}
\newcommand{\Sd}[1][d]{\mathcal{S}^{(#1)}}
\newcommand{\Sthree}{\Sd[3]}
\newcommand{\Stwo}{\Sd[2]}
\newcommand{\Rbn}[1]{\ensuremath{R^{(#1)}_{\beta,n}}}
\newcommand{\Rnd}[1]{\ensuremath{R^{(#1)}_n}}
\newcommand{\Rntwo}{\ensuremath{\Rnd{2}}}
\newcommand{\medm}{\mu}
\newcommand{\R}{\mathbb{R}}
\newcommand{\N}{\mathbb{N}}
\newcommand{\vx}{x}
\newcommand{\vy}{y}
\newcommand{\vz}{z}
\DeclareMathOperator{\Var}{Var}

\usepackage{authblk}

\title{Random triangulations of the $d$-sphere with minimum volume}
\author[1]{Agelos Georgakopoulos\thanks{Supported by EPSRC grants EP/V048821/1 and EP/V009044/1. All authors have been supported by the European Research Council (ERC) under the European Union’s Horizon 2020 research and innovation programme (Grant Agreement No. 639046).}}
\author[2]{John Haslegrave}
\author[3]{Joel Larsson Danielsson}
\affil[1]{University of Warwick, UK}
\affil[2]{Lancaster University, UK}
\affil[3]{Chalmers University of Technology, Sweden}
 
\date{}

\begin{document}
\maketitle

\begin{abstract}
We study a higher-dimensional analogue of the {Random Travelling Salesman Problem}: let the complete $d$-dimensional simplicial complex $K_n^{d}$ on $n$ vertices be equipped with 
i.i.d.\ volumes on its facets, uniformly random in $[0,1]$. 
What is the minimum volume $M_{n,d}$ of a sub-complex homeomorphic to the $d$-dimensional sphere $\mathbb{S}^d$, containing all vertices?  We determine the growth rate of $M_{n,2}$, and prove that it is well-concentrated. For $d>2$ we prove such results to the extent that current knowledge about the number of triangulations of $\mathbb{S}^d$ allows.

We remark that this can be thought of as a model of random geometry in the spirit of Angel \& Schramm's UIPT, and provide a generalised framework that interpolates between our model and the uniform random triangulation of $\mathbb{S}^d$.
\end{abstract}

\section{Introduction}
In this paper we study the following analogue of the \emph{Random Travelling Salesman Problem}, in line with a recent trend of higher-dimensional generalisations of graph-theoretic results.
Let the complete $d$-dimensional simplicial complex $K_n^{d}$ on $n$ vertices be equipped with 
i.i.d.\ costs on its facets, uniform in $[0,1]$. 
What is the minimum volume $M_{n,d}$ of a \emph{spanning} sub-complex of $K_n^{d}$ (that is, containing all vertices) homeomorphic to the $d$-dimensional sphere $\mathbb{S}^d$? (The \emph{volume} of a complex is just the sum of costs of its facets.)  We determine the growth rate of $M_{n,2}$, and prove that $M_{n,2}$ is well-concentrated around its mean. For $d>2$ we prove upper and lower bounds and concentration results, but closing the gaps will require progress on the enumeration of triangulations of $\mathbb{S}^d$, which is notoriously difficult.

\subsection{Background and motivation}
In the graph case, $d=1$, our question  asks for the minimum cost of a Hamilton cycle of the complete graph $K_n$. This problem was studied by Frieze \cite{Fri04}, who showed that it has constant cost and is close to the minimum $2$-factor with high probability.  Subsequently, W\"astlund determined the limiting constant precisely \cite{Was10}. A similar, though simpler, problem asks for the spanning tree of minimum cost in $K_n$. This is known to have expectation $\zeta(3)+o(1)$ \cite{Fri85}. 
For the directed case, Karp \cite{Kar79} proved a relationship to the assignment problem, which Aldous \cite{Ald01} used to establish that the expectation tends to $\zeta(2)=\pi^2/6$.

\medskip
Part of the motivation for studying the case $d\geq 2$ comes from the emerging random geometry.
The influential work of Angel \& Schramm \cite{AnSchrUn} advanced the idea of considering a random triangulation \Rntwo\ of the 2-sphere chosen uniformly among all such triangulations with $n$ vertices, as a model of random planar geometry.  Angel \& Schramm \cite{AnSchrUn} proved that \Rntwo\ converges locally, in the sense of Benjamini \& Schramm \cite{BeSchrRec}. The limit is a random triangulation of $\mathbb{R}^2$, called the Uniform Infinite Planar Triangulation. We expect our minimum-volume spanning sphere to converge in the same sense, and lead to a different model of random geometry. Another well-known convergence result about \Rntwo, by Le Galle \cite{LeGBro} and Miermont \cite{MieBro}, is that it has a scaling limit which coincides with the \emph{Brownian map}, a model of random fractal geometry studied extensively due to connections with mathematical physics  (\cite{GwyRan}).

In fact one can interpolate between our model and the uniform triangulation $\Rntwo$ by considering a Boltzmann distribution with an (inverse) temperature parameter $\beta$. The two extremes, for $\beta=0$ and $\beta=\infty$ of this distribution correspond to \Rntwo\ and our model respectively. It would be very interesting to understand the phase-transition phenomena as $\beta$ varies. We provide concrete definitions and questions in Section~\ref{Outlook}.

\medskip
There has been much previous work on extending notions of (Hamilton) cycles to hypergraphs in an essentially one-dimensional way; see \cite{KOsurvey} for a survey of results in this area. The topological approach we consider here has emerged recently as the focus of several papers \cite{GHMN,ILMPS,KLNS,KPTZ,LNY22,luria-tessler}, and fits naturally into Linial's `higher-dimensional combinatorics' programme, see e.g.\ \cite{LP16}.

Linial \& Meshulam studied the homology of $2$-dimensional random simplicial complexes \cite{LM06}, which was subsequently extended to higher dimensions by Meshulam \& Wallach \cite{MW09}. This triggered a lot of research on random simplicial complexes analogous to work on Erd\H os--R\`enyi random graphs. In particular, Luria \& Tessler \cite{luria-tessler} determined the threshold for the emergence of a spanning $2$-sphere in the Linial--Meshulam model, using a delicate second-moment argument. Moving away from spanning structures, Benjamini, Lubetzky \& Peled determined the minimum cost of a disc triangulation with boundary $u,v,w$ \cite{min-weight-disk-triangulation}, which may be thought of as a $2$-dimensional analogue of the shortest path between two vertices. 

Besides probabilistic results, there has also been an interest in extending extremal results from graphs to higher dimensions topologically, by replacing cycles by triangulations of $\mathbb{S}^d, d\geq 2$. For example, Georgakopolous, Haslegrave, Montgomery \& Narayanan \cite{GHMN} asymptotically determined the minimum codegree guaranteeing a spanning 2-sphere in a $3$-graph. See e.g.\ \cite{ILMPS,KLNS,KPTZ,LNY22} for further results with a similar flavour.

\subsection{Our setup}
Instead of the complex $K_n^{d}$, we will work with the complete $(d+1)$-uniform hypergraph $K_n^{(d)}$ on $n$ vertices. Let $\Sd_n$ be the set of sub-hypergraphs of $K_n^{(d)}$ that form a combinatorial $d$-sphere spanning all $n$ vertices. In other words, $\Sd_n$ is the set of spanning subgraphs $S$ such that the simplicial complex formed by taking the downward closure of $E(S)$ is homeomorphic to $\mathbb{S}^d$. We endow the (hyper)edges of $K_n^{(d)}$ with i.i.d.\ random costs, uniformly distributed in $[0,1]$. We write $W_e$ for the random cost of the edge $e\in E(K_n^{(d)})$, and $W_S$ for the total \emph{cost}, or \emph{volume}, $\sum_{e\in E(S)} W_e$ of a subgraph $S$, and $M(\mathcal{S})$ for the minimum volume $\min_{S\in \mathcal{S}} W_S$ of a family $\mathcal S$ of subgraphs of $K_n^{(d)}$.

We call $\Sd_n$ the set of \emph{spanning spheres}, and refer to the spanning sphere $S\in\Sd_n$ achieving this minimum (which is almost surely unique) as the \emph{minimal spanning sphere}. In this paper, we prove upper and lower bounds on $\E(M(\Sd_n))$. For $d=2$ these match up to a constant factor. For dimensions $2$ and $3$ we also show that this random variable is sharply concentrated around its mean.

Key difficulties in dimensions $d>2$ are the absence of exact asymptotics for the number of spanning spheres, and the fact that the number of facets in a spanning sphere is not fixed.
Enumerating the spheres is trivial in dimension $1$, and is a classical result of Tutte \cite{tutte1962census} in dimension $2$.

\subsection{Our results}
Our main results are concentration inequalities as well as upper and lower bounds for $M(\Sd_n)$ for $d= 2,3$. These together imply sharp concentration: that $M(\Sd_{n})/\mu^{(d)}_n\to 1$ in probability as $n\to \infty$, where $\mu^{(d)}_n$ denotes the median of $M(\Sd_{n})$.

The upper bound holds for $d\geq 4$ as well, but for the lower bounds and concentration inequality we need to make an additional assumption on the number of $d$-spheres with a given number of facets.
Under an even stronger assumption, we prove matching upper and lower bounds as well as concentration for any $d\geq 2$.

Upper bounds are proven using an explicit construction, while lower bounds are proven using a first moment method. For sharp concentration, we use the Talagrand inequality. In dimension 2, however, this turns out to not be sharp enough, and we instead use a recent concentration inequality by the third author~\cite{patch-ineq} which builds on the Talagrand inequality.

\begin{theorem}[$d=2$]
\label{thm:2sphere}
Let $\mu^{(2)}_n$ denote the median of $M(\Stwo_n)$. 
\begin{enumerate}[(i)]
    \item Letting $\alpha:=\sqrt{3^3/4^3e}\approx 0.394$, we have 
    \begin{align*}
    \alpha-o(1)\leq \frac{\mu^{(2)}_n}{\sqrt{n}}\leq \frac{e}{2}\alpha+o(1).
    \end{align*}
    \item With probability at least $1-\exp(-n^{0.02})$,
    \[
    |M(\Stwo_{n})-\mu^{(2)}_n|\leq n^{0.4}.
    \]
\end{enumerate}
\end{theorem}
We remark that the constants of the lower and upper bound of Theorem~\ref{thm:2sphere} (i) differ by a multiplicative factor of only $e/2 \approx 1.36$. 

We also consider a variant of our model where the (independent, uniform) costs are put on the pairs of vertices of $K_n^2$; see \cref{cor} for details.

\medskip
We next consider higher dimensions. For each $d\geq 2$ let $\beta_d$ be a constant such that at most $2^{O(m)} m^{\beta_d m}$ spheres in $\Sd_n$ have $m$ facets, for every $m$; such constants are known to exist but the optimal values are unknown for $d\geq 3$ \cite{howmanytriang:rivasseau}.
\begin{theorem}
\label{thm:dsphere}
The following hold for any $d\geq 2$.
\begin{enumerate}[(i)]
\item There exists a constant $b=b(d)$ such that  $M(\Sd_n)\leq bn^{1-1/d}$ with high probability.
    \item If $\beta_d\leq 1$, then $M(\Sd_n)\geq an^{1-\beta_d}$ for some constant $a=a(d)$, w.h.p.
    \item If $\beta_d<1/2$, then there exists some $c>0$ such that $\Pr(|M(\Sd_n)-\mu^{(d)}_n|>t) \leq e^{-ct^2/n}$ for any $t>0$.
\end{enumerate}
\end{theorem}
It is known that $\beta_4\leq 1$~\cite{howmanytriang:rivasseau}, and so (ii) applies in dimension $4$, but for $d\geq 5$ this is not known. 
The third author has proved that $\beta_3\leq 8/21$~\cite{3-sphere-entropy}. Combining this with the lower bound in \cref{LB:general-graph-family} and the upper bound and concentration in \cref{thm:dsphere} above, we obtain the following.

\begin{theorem}[$d=3$]
\label{thm:3sphere} There are constants $c,C>0$ such that for all sufficiently large $n$,
\[
cn^{13/21}\leq \mu^{(3)}_n\leq Cn^{2/3},
\]
and furthermore
\[
|M(\Sthree_{n})-\mu^{(3)}_n|=\Op(\sqrt{n}).
\]
In particular, $M(\Sthree_{n})/\mu^{(3)}_n\to 1$ in probability as $n\to \infty$. 
\end{theorem}
All lower bounds are proved using a first moment method. For this technique to work, we need good upper bounds on the number of $n$-vertex spheres with a given number of facets $m$. (We discuss this in \cref{section:howmany}.) For fixed $n$, while spheres with a large number of facets tend to have larger cost, there can be many more such spheres than spheres with few facets. Thus, depending on how quickly the number of spheres grow with $m$, the expected number of `cheap' spheres might be dominated by spheres with few or many facets.

\medskip
A well-known question of Gromov~\cite{gromov-spaces-questions} asks whether the number of isomorphism classes of spheres in $\mathcal{S}_n^{(d)}$ with $m$ facets grows at most exponentially fast as a function of $m$. If so, we would be able to prove much stronger results. Furthermore, if we require our sphere to lie in a suitable subfamily $\mathcal{P}\subseteq \mathcal{S}_n^{(d)}$ for which this holds, then our results apply to the minimum-cost sphere in this family; for such families, the minimum cost tends to be achieved by a sphere with close to the minimum number of facets. Here it is natural to restrict our attention to \textit{symmetrical} families, that is, those that are unions of isomorphism classes. We also require a specific sphere $S^*_{n,d}$ of a particularly simple form, used in proving our upper bound, to lie in the family. We define $S^*_{n,d}$ as follows: its facets are precisely those $(d+1)$-sets whose intersection with $\{v_1,\ldots,v_{n-d}\}$ lies in $\{\{v_1,v_2\},\ldots,\{v_{n-d-1},v_{n-d}\},\{v_{n-d},v_1\}\}$. We will subsequently show that this is indeed a sphere.

We prove matching upper and lower bounds of order $n^{1-\frac{1}{d}}$, and sharp concentration, for such $\mathcal{P}$ in any dimension $d$.

\begin{theorem}
\label{LC-sphere-scaling}
Assume $\mathcal{P}\subseteq \mathcal{S}_n^{(d)}$ is a symmetrical subfamily including $S^*_{n,d}$ and containing at most $K^m$ isomorphism classes of spheres with $m$ facets for some $K$ and every $m$.
Then there is a $\mu=\Theta(n^{1-\frac{1}{d}})$ such that for some constant $c>0$ and any $0<t\leq \sqrt{n}$,
\[
\Pr\big(|M(\mathcal{P})-\mu|>t\sqrt{n}\big) \leq e^{-ct^2},
\]
and the optimal spanning sphere has at most $(1+\eps)dn$ facets with high probability (for any $\eps>0$).
\end{theorem}
In particular, \cref{LC-sphere-scaling} holds when $\mathcal{P}$ is the subfamily of locally constructible spheres, the definition and motivation of which is discussed in \Cref{section:lc-spheres}. Furthermore, if Gromov's aforementioned question has a positive answer, then \cref{LC-sphere-scaling} holds for $\mathcal{P}=\mathcal{S}_n^{(d)}$.

\subsection{Outline of the paper}
In Section \ref{section:howmany} we discuss bounds on the number of spheres, which will be relevant to Theorems \ref{thm:dsphere} and \ref{thm:3sphere}. We then proceed to prove the upper and lower bounds constituting Theorem \ref{thm:2sphere} (i) in Section \ref{sec:bounds} and Theorem \ref{thm:dsphere} (i) and (ii) in Section \ref{section:d-bounds}. We prove Theorem \ref{thm:2sphere} (ii) in Section \ref{sec:2-conc} and Theorem \ref{thm:dsphere} (iii) in Section \ref{section:concentration}. Theorem \ref{thm:3sphere} follows from our other results as detailed above. We discuss locally constructible spheres, and prove Theorem \ref{LC-sphere-scaling}, in Section \ref{section:lc-spheres}.

\section{Definitions \& preliminaries} \label{section:defs}

\subsection{Hypergraphs and simplicial spheres}
A \emph{hypergraph} $G=(V,H)$ consists of a set $V$ of \emph{vertices}, and a set $H$ of pairwise distinct subsets of $V$, called \emph{(hyper)edges}. If each element of $H$ has the same number of vertices $k$, we say that $G$ is $k$-uniform. Let $K_n^{(d)}$ denote the hypergraph with vertex set $[n]$,  where each $(d + 1)$-tuple of vertices forms a hyperedge.

A \emph{simplicial $d$-sphere} is a triangulation of $\mathbb{S}^d$, i.e.\ a simplicial complex $S$ homeomorphic to $\mathbb{S}^d$. We will sometimes drop the term \emph{simplicial} for brevity. We can realise $S$ as a hypergraph, by using the same vertex set, and declaring the set of vertices of each facet of $S$ to be a hyperedge. All the simplicial $d$-spheres we will consider in this paper are realised as spanning sub-hypergraphs of $K_n^{(d)}$, i.e.\ hypergraphs on $[n]$ with set of hyperedges contained in that of $K_n^{(d)}$. We define a  \emph{simplicial $d$-ball} analogously.

\subsection{Sums of uniform random variables}

For $p>0$, we let $U(0,p)$ denote a random variable uniformly distributed in $[0,p]$.
We will be using the following straightforward bound on the sum of uniform random variables.
\begin{lemma}
\label{claim:uniforms-sum}
Let $X_1,\ldots, X_m$ be  i.i.d.\  $U(0,1)$ random variables. Then
\[\Pr\left(\sum_{i=1}^m X_i\leq L\right) \leq L^m/m!\leq (L e/m)^m.\]
\end{lemma}
\begin{scproof}
Let $\chi_A$ be the characteristic function 
of the set $A:={ \{x\in \mathbb{R}_+^m: \sum_{i=1}^{m}x_i \leq L\}}$. Then 
\[
\Pr\left(\sum_{i=1}^m X_i\leq L\right)=\int_{[0,1]^{m}} \chi_A d\mu \leq \int_{[0,L]^{m}} \chi_A d\mu ={L^m}/{m!},
\]
where $\mu$ is the Lebesgue measure on $\mathbb{R}^m$.
\end{scproof}

\section{On the number of simplicial spheres}
\label{section:howmany}
In \cref{section:lowerbound}, we give a generic lower bound on the random variable $M(\Sd_{n})$, using a first moment argument by finding an $L\in \R$ such that the expected number of $d$-spheres with cost at most $L$ is $o(1)$. However, for this to work we need a good enough upper bound on the total number $A^{(d)}_{n,m}$ of labelled simplicial $d$-spheres with a given number of vertices $n$ and facets $m$. Previous bounds often in fact consider spheres up to vertex permutations (often called combinatorially distinct spheres); we will refer to these as \emph{isomorphism classes} and denote their number by $B^{(d)}_{n,m}$. On the one hand, clearly $A^{(d)}_{n,m}\leq n!B^{(d)}_{n,m}$. But on the other hand, since each sphere has at most $(d+1)!m$ automorphisms -- an arbitrary facet can be mapped to any of $m$ facets in any of $(d+1)!$ orientations, and this determines the automorphism -- we also have that $n!\,B^{(d)}_{n,m}\leq (d+1)!\, m A^{(d)}_{n,m}$. Hence $A^{(d)}_{n,m}=m^{O(1)} n!\,B^{(d)}_{n,m}$.

For $d=2$ we have $m=2n-4$,\footnote{Indeed, Euler's formula implies that each simplicial $2$-sphere with $n$ vertices has exactly $m=2n-4$ faces and $3m/2=3n-6$ edges.} and so these numbers are merely functions of $n$. Tutte \cite{tutte1962census} determined the number of isomorphism classes, which is growing exponentially with $n$:
\begin{equation}
  B_{n,2n-4}^{(2)} =(1+o(1)) \frac{1}{16} \sqrt{\frac{3\pi}{2}}n^{-5/2} \left(\frac{4^4}{3^3}\right)^{n+1}=2^{O(n)}.  
  \label{tutte-enum}
\end{equation}
Thus, for sufficiently large $n$, 
\begin{equation}A_{n,2n-4}^{(2)}\leq n! B_{n,2n-4}^{(2)}\leq (4^4n/3^3 e)^n.\label{tutte-labelled}\end{equation} 
A first moment argument using \eqref{tutte-labelled} (see \cref{LB:general-graph-family}) provides a lower bound on $M(\Stwo_{n})$ within a constant factor of the upper bound we prove in \cref{section:upperbound}. That upper bound can be slightly improved (while still not matching the lower bound) as an easy corollary of a theorem on the related threshold problem by Luria \& Tessler~\cite{luria-tessler}. For details, see the proof of \cref{thm:2sphere}.

In higher dimensions, the number $m$ of facets varies among simplicial spheres with $n$ vertices. There is a trade-off where, for a fixed $n$, there are many more spheres with large $m$, but these tend to have higher cost and so are individually much less likely to have cost at most $L$. If the second factor dominates the former, the minimum-cost sphere typically has few facets. In order to understand this interplay we need bounds on the asymptotic growth rate of $A^{(d)}_{n,m}$ as $m$ becomes large. 

Our main interest, in view of \Cref{thm:dsphere}, is how small we can take $\beta_d$ such that $A^{(d)}_{n,m}\leq 2^{O(m)}m^{\beta_d m}$. However, previous work in this area has concentrated on the growth of the number of isomorphism classes of simplicial spheres with $m$ facets and any number of vertices. This is unknown for $d\geq 3$. In fact, while it is easy to show that there are $2^{\Omega(m)}$ spheres with $m$ facets, it is a major open question of Gromov~\cite{gromov-spaces-questions} whether their number grows at most exponentially fast. 

Upper bounds on the number of spheres with $n$ labelled vertices and any number of facets, of roughly $\exp({n^{\lceil d/2\rceil+o(1)}})$ have been obtained \cite{Kal88}, but these are insufficient for our purposes, as they give no control over the relationship between $n$ and $m$, and are dominated by triangulations with $n\ll m$. Previous bounds on the number of isomorphism classes of spheres with $m$ facets are of the form $m^{\gamma_d m}$ for some constants $\gamma_d$ \cite{howmanytriang:rivasseau}. Since this also gives an upper bound on the number of isomorphism classes with $m$ facets and exactly $n$ vertices, and labelling the vertices adds a multiplicative factor of at most $n!\ll m^{m/d}$, we obtain a bound of the desired form with $\beta_d=\gamma_d+1/d$. Unfortunately, the best known upper bounds on $\gamma_d$ are too large for our purposes when $d\geq 3$. For instance, it was shown in~\cite{howmanytriang:rivasseau} that $\gamma_3\leq 1/3$ and $\gamma_4 \leq 3/4$, and these bounds are low enough to imply that $M(\Sthree_{n})\gg 1$ and $M(\Sd[4]_n)=\Omega(1)$ (as we will show in \cref{section:lowerbound}). However, these are neither anywhere near our upper bound $n^{1-1/d}$ in \cref{section:upperbound}, nor large enough for our upper bound on the variance $\Var(M(\Sd_{n}))=O(n)$ in \cref{section:concentration} to imply sharp concentration. Moreover, for $d\geq 5$ the best known upper bounds have $\gamma_d>1$, which render our first moment method of \cref{section:lowerbound} ineffective.

In a forthcoming paper, the third author improves on the upper bound in the case $d=3$. This increases the lower bound on $M(\Sd_{n})$ enough for the argument in \cref{section:concentration} to imply sharp concentration, yet not quite enough to match the upper bound in \cref{section:upperbound}.
\begin{theorem}[\cite{3-sphere-entropy}]
\label{entropy-bound}
For sufficiently large $n$, we have $A^{(3)}_{n,m}\leq m^{\frac27(m+n)}$.
\end{theorem}
Since in this case $n\leq m/3+O(1)$, it follows that $\beta_3\leq \frac27(1+1/3)=8/21$.

In \cref{section:lc-spheres} we discuss a subfamily of more combinatorially tractable spheres, called the \emph{locally constructible} (LC) spheres. All $2$-spheres are LC, but for every $d\geq 3$ there are $d$-spheres that are not. It has been shown~\cite{LC-spheres-definition} that there are only exponentially many LC-spheres with a given number of facets, which allows us to prove much stronger results for this class.

\section{Upper and lower bounds for 2-spheres} \label{sec:bounds}

In this section we provide simple  bounds that establish \cref{thm:2sphere}(i).
\begin{scproof}[Proof of \cref{thm:2sphere}(\textnormal{i})]
For the lower bound, fix $L\in \R$, and define the random variable $X$ to be the number of $S\in \Stwo_n$ whose cost $W_S$ is at most $L$.

By Tutte's formula \eqref{tutte-enum}  there are $o(\gamma^n n!)$ simplicial $2$-spheres on $n$ vertices, where $\gamma:=4^4/3^3$. Each simplicial $2$-sphere has $m=2n-4$ faces, hence by Lemma~\ref{claim:uniforms-sum} it has cost at most $L$ with probability at most $(L e/m)^m=n^{O(1)}(Le/2n)^{2n}$. By linearity of expectation we thus have 
\begin{align} \label{EX}
    \E (X) \leq \gamma^n n! n^{O(1)} (Le/2n)^{2n}
=n^{O(1)}(L^2e\gamma /4n)^n.
\end{align}

Let $L:= (1-\eps)\sqrt{4n/e\gamma}$ for some $\eps>0$, so that the second factor is $(1-\eps)^{2n}$. It follows from the Markov inequality that 
$M(\Stwo_n)>  L$, and hence with high probability \[M(\Stwo_n)> (1-\eps)\sqrt{\frac{3^3n}{4^4e}}.\]

\medskip
For the upper bound, we use the following result from~\cite{luria-tessler}: there is a sharp threshold for the appearance of a spanning $2$-sphere in the Linial--Meshulam model\footnote{This is the natural generalisation of the binomial random graph to higher dimensions: given some paramenter $p\in [0,1]$, and a dimension $d\in \mathbb{N}$, we keep each $d$-simplex with vertices in $[n]$ independently with probability $p$.} at $p_c:=\sqrt{e/\gamma n}$. Letting $p:=(1+\eps)p_c$, we conclude that $M(\Stwo_n)\leq mp$, because one way to sample the Linial--Meshulam model is to assign i.i.d.\ $U(0,1)$ random costs to the 2-simplices, and delete those of cost greater than $p$. We can, however, improve this bound by a factor of $2$ as follows.

Let $R\subseteq \Stwo_n$ be the random set of $2$-spheres with no face having cost more than $p$.
For any $S\in \Stwo_n$, conditionally on $\{S\in R\}$, the conditional distribution of the costs of the faces in $S$ are i.i.d.\ $U(0,p)$. Thus the conditional distribution $(W_S\mid S\in R)$ is the sum of $m$ i.i.d.\ random variables with expected value $p/2$, and hence Hoeffding's inequality yields $\Pr(W_S>(1+\eps)mp/2\mid S\in R)\leq e^{-\eps^2m/2}$.

Thus whenever $R$ is non-empty, an arbitrary $S\in R$ has cost at most $(1+\eps)mp/2$ w.h.p., and $R$ is non-empty w.h.p.\ since $p>(1+\eps)p_c$. Hence $M(\Stwo_n) $ is at most $ (1+\eps)mp/2\leq (1+\eps)^2\sqrt{ne/\gamma}$ w.h.p. Since $M(\Stwo_n) \leq m$ with the remaining probability, the result follows.
\end{scproof}

\medskip
Next, we obtain a result analogous  to \cref{thm:2sphere}(i) for the variant of our model where we put costs on the $1$-cells (edges) instead of $2$-cells:  we define $N(\Stwo_n)$ just like $M(\Stwo_n)$ 
except that we put independent $U(0,1)$ costs on the pairs of vertices of $K_n^{(d)}$, and define the cost of $S\in \Stwo_n$ to be the sum of the costs of the pairs appearing in the edges (i.e.\ triples of vertices) of $S$. 

\begin{theorem} \label{cor}
Let $\nu^{(2)}_n$ denote the median of $N(\Stwo_n)$. There is a constant $C\in \R$ such that 
    \begin{align*}
    \frac{3^2}{4^{4/3}e^{2/3}} -o(1)\leq \frac{\nu^{(2)}_n}{{n^{2/3}}}\leq C.
\end{align*}    
Furthermore, $N(\Stwo_n)$ is sharply concentrated around $\nu^{(2)}_n$: for some small fixed $\delta>0$,
\[\Pr(|N(\Stwo_n)-\nu^{(2)}_n|\geq (\nu^{(2)}_n)^{1-\delta})\leq \exp(-n^\delta).\]
\end{theorem}

\begin{scproof}
We follow the lines of the proof of \cref{thm:2sphere}(i). Each simplicial $2$-sphere with $n$ vertices has exactly $3n-6$ edges, thus again by Lemma~\ref{claim:uniforms-sum} it has cost at most $L$ with probability at most $n^{O(1)}(Le/3n)^{3n}$. The calculation of \eqref{EX} now becomes
\[\E (X) \leq n^{O(1)} \gamma^n n!  (Le/3n)^{3n}
\approx \left(\frac{\gamma n}{e} \frac{L^3 e^3 }{3^3 n^3}\right)^{n} = \left(\frac{\gamma L^3 e^2 }{3^3 n^2}\right)^{n} = \left(\frac{4^4 L^3 e^2 }{3^6 n^2}\right)^{n}.\]

Choosing $L$ so that this equals $(1-\eps)^n$, it follows as above that, with high probability,  
$N(\Stwo_n)> (1-\eps)(\frac{3^2}{4^{4/3}}) (n/e)^{2/3}$.

For the upper bound, we use an explicit construction.
The proof idea is to start with a $2$-sphere on $4$ vertices (the boundary of a tetrahedron) and $n-4$ free vertices. We repeatedly make barycentric subdivisions of its faces until the sphere has $n$ vertices.

A heuristic argument for the growth rate $N(\Stwo_n)\sim n^{2/3}$ is as follows. When there are $i$ vertices in the construction so far, we must choose one of the remaining $n-i$ free vertices to use for the barycentric subdivision. The cost of picking any given vertex is a sum of three i.i.d.\ $U(0,1)$ random variables. The expected minimum over $n-i$ such random variables is of order $(n-i)^{-1/3}$, and summing over $i<n$ the total expected cost is of order $n^{2/3}$. We want to bound the upper tail of this sum.

However, in order to make this heuristic rigorous we need to avoid problems of dependencies between the costs of the edges added in the current barycentric subdivision and the previous ones. For this we use two tricks: First, similarly to the proof of \cref{thm:2sphere}(ii), we partition the vertex set $[n]$ into four sets $\{V_i\}_{i=1}^4$. For the sake of convenience, we assume that a vertex $j\in [n]$ lies in $V_i$ if and only if $j\equiv i \mod 4$. Next, we choose the faces to perform barycentric subdivisions of in a particular way. Starting with a tetrahedron $v_1v_2v_3v_4:=1234$, make a barycentric subdivision of the face $v_2v_3v_4$ using some vertex $v_5$, then a barycentric subdivision of the face $v_3v_4v_5$ using some vertex $v_6$, and so on. This $2$-sphere can be encoded by the word $v_1v_2\ldots v_n$, and the set of edges in its $1$-skeleton are $\{v_iv_j:|i-j|\leq 3\}$.
(In hypergraph terms, it corresponds to the shadow of a tight path in a $4$-uniform hypergraph.) Furthermore, we pick each such $v_i$ from the $V_j$ which has $i\equiv j \mod 4$. (That is, $v_5\in V_1$, $v_6\in V_2$ and so on.)

For $i\geq 5$, let $Z_{i,u}:=X_{v_{i-3}u}+X_{v_{i-2}u}+X_{v_{i-1}u}$ denote the cost of the edges added when using the vertex $u$ for a barycentric subdivision of $v_{i-3}v_{i-2}v_{i-1}$, and $Z_i:=\min_{u\in V_i}Z_{i,u}$ (where the indices for $V$ are taken modulo $4$). For $i\leq 4$, define $Z_1=0$, $Z_2=X_{12}$, $Z_3=X_{13}+X_{23}$ and $Z_4=X_{14}+X_{24}+X_{34}$, so that $\sum_{i=1}^4 Z_i\leq 6$ is the total cost of the edges in the tetrahedron $1234$.
Note that  $v_1v_2\ldots v_n$ encodes a $2$-sphere with cost $\sum_{i=1}^n Z_i\geq N(\Stwo_n)$. 
\begin{claim}
The random variables $Z_i$ are independent.
\end{claim}
\begin{proof}[Proof of claim]
To prove the claim, consider an edge $ww'\in \binom{[n]}{2}$. How many times is its cost $X_{ww'}$ explored?
Assume without loss of generality that $w$ is added to the growing sphere before $w'$, and $w=u_i$ for some $i$. Let $k\in \{0,1,2,3\}$ be such that $w\in V_i,w'\in V_{i+k}$ (where indices are again taken modulo $4$). If $k=0$, the edge $ww'$ is never considered -- the final $2$-sphere contain no edge with both endpoints in the same set $V_i$. If $k=1,2$ or $3$, the edge $ww'$ is explored only when choosing between candidates for $v_{i+k}:=\arg \min_{u\in V_{i+k}}Z_{i+k,u}$. So in either case, $ww'$ is only explored at most once, and hence the $Z_i$'s are independent.
\end{proof}

It only remains to show that $\sum_{i=1}^n Z_i\leq Cn^{2/3}$ whp for some constant $C$. Note that the random variables $Z_i$ are not identically distributed, but that the distributions of $Z_i$ and $Z_j$ only differ slightly if $i$ and $j$ are close.
We'll therefore group the terms into at most $k$ blocks with $k$ terms each (picking $k$ as the smallest multiple of $4$ such that $k^2\geq n$), and separately upper bound the sum of the $Z_i$'s in each block.
More precisely, we partition $\{1,2,\ldots ,n\}$ into sets $\{A_j\}_{j=0}^k$ by setting $A_j:=\{s\in \N:jk\leq n-s< {(j+1)}k\}$. Note that $|A_j|\leq k$ for all $j$ (in fact $|A_j|=k$ for all but a bounded number of $j$).
Let $Y_j:=\sum_{i\in A_j}Z_i$, so that $\sum_{i=1}^n Z_i=\sum_{j=0}^k Y_j$. 

For $i\in A_j$, $j\geq 1$, each random variable $Z_i$ is the minimum of at least $jk/4$ independent random variables $Z_{i,u}$.
Applying \cref{claim:Chernoff-sum-min-sum} to $Y_j$ with $k_0=|A_j|\leq k$, $n_0=jk/4$ and $b=25k$, we get that
\begin{equation}
\label{ineq:blockprob}
\Pr\Big( Y_k\geq \underbrace{25k (jk/4)^{-1/3}}_{=:a_j}\Big)\leq e^{-k}.
\end{equation}
Noting that $\sum_{j=1}^{k}j^{-1/3}< (3/2+o(1))k^{2/3}$ by an integral comparison and recalling that $k=(1+o(1)) \sqrt{n}$, the sum $\sum_{j=1}^{k}a_j $ is at most $ctn^{2/3}$ for some constant $c<60$.
Next, for $j=0$, note that $Y_0$ is the sum of the $Z_i$'s in the \emph{last} block. Hence $Y_0\leq |A_0|=k$, because $Z_i\leq 1$ surely. In other words, the inequality (\ref{ineq:blockprob}) holds for $k=0$ with $a_0:=k$.
Then $\sum_{j=0}^{k}a_j< 60n^{2/3}$ for $n$ sufficiently large. By a union bound,
\[
\Pr\left(\sum_{i=1}^{n}Z_i \geq 60n^{2/3}\right)\leq ke^{-k}= e^{-\Omega(\sqrt{n})}.
\]
Recall now that $\sum_{i=1}^{n}Z_i$ was the cost of the sphere found with our explicit construction, which is an upper bound on the optimal cost. So with high probability, $N(\Stwo_n)\leq 60n^{2/3}$.

\medskip
For concentration, we apply \cref{thm:conc-general}. There are at most $K^n n^n$ $2$-spheres on $n$ vertices, for some constant $K$, and each has $3d-6$ edges. So \cref{thm:conc-general} gives, with  $\beta=1/3$ and $t=3$, that there exists a small constant $\delta>0$ such that
\[\Pr(|N(\Stwo_n)-\nu|\geq \nu^{1-\delta})\leq \exp(-n^\delta).\qedhere\]
\end{scproof}

\section{Bounds for \texorpdfstring{$d$}{d}-spheres}
\label{section:d-bounds}
\subsection{Upper bound}
\label{section:upperbound}

In this section we show that there exists a sufficiently cheap sphere with high probability; in order to fulfil the requirements of \cref{LC-sphere-scaling} we will show this is true even when restricted to spheres isomorphic to $S^*_{n,d}$. Recall that the facets of $S^*_{n,d}$ are the $(d+1)$-sets whose intersection with $\{v_1,\ldots,v_{n-d}\}$ lies in $\{\{v_1,v_2\},\ldots,\{v_{n-d-1},v_{n-d}\},\{v_{n-d},v_1\}\}$. For $d=2$, this is easy to visualize, being the boundary of a bipyramid. In general, it consists of an $n-d$ cycle and $d$ poles; in fact, we consider the subclass $\Theta_d$ of isomorphic spheres where the cycle vertices may be permuted, but the poles remain $v_1,\ldots, v_d$. That is, we take an arbitrary cycle $C$ on vertex set $v_{d+1},\ldots,v_n$. For each such cycle $C$, define $f_{xy}=\{v_1,\ldots,v_d,x,y\}$ for each $xy\in E(C)$, let $B(C)$ be the closure of $\{f_{xy}\}_{xy\in E(C)}$, and let $S_C:=\partial B(C)$. Let $\Theta_d$ be the family of all such $S_C$.
\begin{proposition}
\label{its-a-sphere}
$S_C$ is a locally constructible simplicial $d$-sphere with $d(n-d)$ facets, given by
\[
\{f_{xy}-\{z_i\}: xy\in E(C), 1\leq i \leq d\}.
\]
In particular, taking $C=v_{d+1}v_{d+2}\cdots v_nv_{d+1}$, we have $S_C=S^*_{n,d}$ is an LC $d$-sphere.
\end{proposition}
Since it requires the definition of locally constructible, we defer the proof of \cref{its-a-sphere} to \cref{section:lc-spheres}.

\begin{proposition}
\label{UB:polar-sphere}
For $d\geq 2$, there exist $c_d>0$ such that there exists an $S\in\Sd_n$ with cost at most $c_d n^{1-\frac{1}{d}}$, and $d(n-d)$ facets, with high probability.\footnote{For $d=2$ \cref{sec:bounds} provides a better bound that relies on \cite{luria-tessler}.}
\end{proposition}

\begin{scproof}
We will apply the first moment method to the family $\Theta^d$ defined above.

\noindent Our aim is now to choose a Hamilton cycle $C$ of a complete graph $G$ equipped with suitable edge costs on vertex set $\{v_{d+1},\ldots,v_n\}$ such that the resulting $d$-sphere $S_C$ has low cost $W_{S_C}$. 
The contribution to $W_{S_C}$ of the edge $xy$ is  $w_{xy}:= \sum_{i=1}^d W_{f_{xy}-\{z_i\}}$, and we assign these $w_{xy}$ as the edge-costs on $G$.
Note that the $w_{xy}$ are i.i.d., and that the total cost of the sphere $S_C$ is $\sum_{xy\in E(C)} w_{xy}$, 
so this is an instance of the traveling salesman problem on the complete graph with i.i.d.\ edge costs.

These edge costs are of so-called \emph{pseudo-dimension} $d$, by which we mean that $\Pr(w_{xy}\leq t)/t^d\to c$ as $t\to 0$ for some constant $c\in (0,\infty)$. (In fact, $\Pr(w_{xy}\leq t)=t^d/d!$ for $t\leq 1$.)
In~\cite{wastlund2009replica} (an expanded version of \cite{wastlundmatching}), W\"{a}stlund studied the traveling salesman problem on $K_n$ with i.i.d.\ edge costs of pseudodimension $d>1$. He showed that the minimum cost of a Hamilton cycle is sharply concentrated around $(1+o(1))c^{-1}\beta_{\mathrm{TSP}}(d) n^{1-1/d}$, where $c$ is as above, and $\beta_{\mathrm{TSP}}(d)$ is some constant only depending on $d$. (We have $K_{n-d}$ instead of $K_{n}$, but the asymptotics do not change.)  The proposition follows by picking any $c_d>d! \beta_{\mathrm{TSP}}(d)$.
\end{scproof}

\begin{remark}
Instead of using the result in \cite{wastlund2009replica}, a more elementary proof method is to analyze a greedy algorithm for finding a Hamilton cycle, and apply a Chernoff bound similar to \cref{claim:Chernoff-sum-min-sum}. The resulting upper bound is only worse by a small constant factor. 
\end{remark}

\subsection{Lower bound}
\label{section:lowerbound}

Here we give a first moment lower bound on $M(\F)$, similar to that of \cref{sec:bounds}, for a family $\F$ of subgraphs of $K_{n}^{(d)}$, given some bound on the size of $\F$ and the sizes of members of $\F$. We make no other assumptions on the structure of $\F$. 

\begin{proposition}
\label{LB:general-graph-family}
Let the hyperedges of $K_{n}^{(d)}$ be equipped with i.i.d.\ $U(0,1)$-costs. Suppose $\F$ is a family of subgraphs of $K_{n}^{(d)}$ such that
\begin{enumerate}
\item each $F\in\F$ has at least some $m_0$ edges, and
\item $  \big|\big\{ F\in \F:|E(F)|=m\big\}\big|\leq b^m m^{\beta m}$ for some constants $b > 0$ and $0<\beta\leq 1$,
\end{enumerate}
and let $c>0$ be an arbitrary constant. Then, there exists an $r=r(b,c)$ such that $M(\F) \geq rm_0^{1-\beta}$ with probability at least $1-e^{-cm_0}$.
\end{proposition}
\begin{remark}
\label{remark:LB-d-sphere}
Every $d$-sphere on $n$ vertices has at least $m_0= dn-(d-1)(d+2)$ facets (see e.g.~\cite{kalai-LBT}, but note that our $d$ corresponds to  $d-1$ there).
Part (ii) of \cref{thm:dsphere} is a special case of \cref{LB:general-graph-family} with $\F=\Sd_n$ and $\beta$ such that $A_{n,m}\leq 2^{O(m)}m^{\beta m}$.
\end{remark}

For a set $F\in \F$ of $m$ edges, the probability that their cost $W_F$ is below a given value $L=\Theta(n^{\alpha})$ (for some constant $\alpha<1$) decays superexponentially fast as a function of $m$. On the other hand, the number of $F\in \F$ with $m$ edges might increase superexponentially fast. If $\alpha+\beta\leq 1$, then the former decay rate outperforms the latter growth rate, so that $F$'s with few edges dominate the expected number of `cheap' subgraphs.
 
\begin{scproof}
We will apply the first moment method to the number of `cheap' $F\in \F$. 
For some $L=L(n)$ to be determined later, let $X_m$ be the (random) number of graphs $F\in \F$ with precisely $m$ facets and with $W_F\leq L$. 
Then $\E X_m \leq b^m m^{\beta m}\cdot \Pr(W_F\leq L)$, where $F$ is an arbitrary subgraph with $m$ edges.

By \cref{claim:uniforms-sum}, $\Pr(W_F\leq L) \leq (L e/m)^m$. Let $L:= \frac{\delta}{eb}\cdot m_0^{1-\beta}$, for some $\delta\in (0,1)$. Since $\beta\leq 1$, we have $m_0^{1-\beta}\leq m^{1-\beta}$, so that
\(
\E X_m \leq \left(bm^\beta Le/m\right)^m\leq \delta^m
\).
By Markov's inequality, $\Pr(M(\F)\leq L)  $ is then at most $\sum_{m\geq m_0}\E X_m\leq \sum_{m\geq m_0}\delta^{m}<\delta^{m_0}/(1-\delta)$. This is at most $e^{-cm_0}$ if we pick $\delta=\delta(c)$ sufficiently small. 
\end{scproof}

\section{Concentration for 2-spheres}
\label{sec:2-conc}

To prove concentration in the case $d=2$, i.e.\ part (ii) of \cref{thm:2sphere}, we will use a concentration inequality from \cite{patch-ineq}, for which we need the following definitions.\footnote{These definitions are given in \cite{patch-ineq} for general families of sub-hypergraphs, but we present them only for the family $\Sd_n$ and subsets of it.}
For any $\F\subseteq \Sd_n$ and sub-hypergraphs $G,P$ of $K_n^{(d)}$, we say that $P$ is a \emph{$G$-patch} if $G\cup P$ contains an $S\in \F$. We define $\rho(G)=\rho_\F(G)$ as the minimum number of edges in a $G$-patch:
\[
\rho_\F(G):=\min\{e(P):P\textrm{ is a $G$-patch}\}=\min\{e(S-G): S\in \F\}.
\]
We think of $G$ as having $\rho_\F(G)$ many `holes', that are `patched' with $P$. Given a sub-hypergraph $G$, let the random variable $\mathrm{Patch}_\F(G)$ be the minimum cost of a $G$-patch, i.e.\ $\mathrm{Patch}_\F(G):= \min\{W_P: P \textrm{ is a }G\textrm{-patch}\}$.

\begin{definition}
We say that a sub-hypergraph $G$ of $K_n^{(d)}$ is $(\lambda,\eps)$-\emph{patchable} (with respect to the random costs $W_e$ and family $\F$) if $\Pr(\mathrm{Patch}_\F(G)>\lambda)<\eps$.
The family $\F$ is said to be $(r,\lambda,\eps)$-\emph{patchable} if every $G$ with $\rho_\F(G)\leq r$ is  $(\lambda,\eps)$-patchable.
\end{definition}
We let $\ell(\F)$ denote the maximum size of a member of $\F$, i.e.\ 
$\ell(\F) := \max\{e(F):F\in \F\}$, where $e(F)$ denotes the number of facets in $F$.

\begin{theorem}[{\cite[Theorem 3.1]{patch-ineq}}, special case $q=1$.]
\label{thm:patchability-ineq}
Let  $\F$ be a family of sub-hypergraphs of $K_n^{(d)}$, and
let $Q$ be the quantile function of $M(\F)$, i.e.\ $\Pr(M(\F)\leq Q(p))=p$ for all $p\in (0,1)$.
If $\F$ is $(k,\lambda,\eps)$-patchable for some $\lambda$, $\eps>0$ and $k\geq \sqrt{8\log(\eps^{-1})\cdot \ell(\F)}$, then with probability at least $1-2\eps$,
\begin{equation}
\label{ineq:patch}
    M(\F)\leq \left(\sqrt{Q(\eps)}+\sqrt{\lambda} \right)^{2} .
\end{equation}
In particular, if $\medm$ is the median of  $M(\F)$, and we have $\eps<1/4$ and $\lambda\leq \medm$, then
with probability at least $1-3\eps$ both $M(\F)$ and $\mu$ lie in the interval $[Q(\eps),(\sqrt{Q(\eps)}+\sqrt{\lambda})^{2}]$, and hence
\[
\lvert M(\F)-\medm\rvert \leq 3 \sqrt{\lambda \medm}.
\]
\end{theorem}

In order to prove part (ii) of \cref{thm:2sphere} using \cref{thm:patchability-ineq}, we need to establish that $\Stwo_n$ satisfies a patchability condition. This is the main effort towards part (ii), which  will follow easily from the following lemma.
\begin{lemma}
\label{lemma:2-patch}
Suppose $k: \N \to \N$ satisfies  $k=\Omega(\sqrt{n})$ and $k=o(n)$. Then for any $r>0$ there is a constant $A=A(r)$ such that 
$\Stwo_n$ is $(k,Ak^{\frac{3}{4}}/n^{\frac{1}{12}},r^{k})$-patchable.

In particular for $r=1/2$ and $k=\Theta(n^{0.51})$, there is a $\lambda=o(n^{0.3})$ such that for any subgraph $H\subseteq K_n^{(3)}$ with $\rho(H)\leq k$,
\[\Pr(\mathrm{Patch}(H)\geq \lambda)\leq 2^{-k}.\]
\end{lemma}

Before proving \cref{lemma:2-patch}, let us see how it implies part (ii) of \cref{thm:2sphere}:

\begin{scproof}[Proof of \cref{thm:2sphere}(\textnormal{ii})] 
Recall that a $2$-sphere with $n$ vertices has $2n-4$ faces, or in other words $\ell(\Stwo_n)=m=2n-4$.
By \cref{thm:2sphere}(i), the median $\medm$ of $M(\Stwo_n)$ satisfies $\medm=\Theta(\sqrt{n})$.

Let $\eps=\exp(-n^{0.02})/3$, and $k=\sqrt{8m\log(\eps^{-1})}=\Theta(n^{0.51})$.
\cref{lemma:2-patch}, applied with $r=1/2$, implies that $\Stwo_n$ is $(k,\lambda,2^{-k})$-patchable for some $\lambda=o(n^{0.3})$. Since $2^{-k}\ll \eps$ and $\lambda\ll \mu$, the assumptions of \cref{thm:patchability-ineq} are satisfied with $\mu,\lambda, k$ and $\eps$ as above, and we obtain 
\[
|M -\medm|\leq 3\sqrt{\medm \lambda}=o(n^{0.4}),
\]
with probability at least $1-3\eps= 1-\exp(-n^{0.02})$. 
\end{scproof}

Thus it only remains to prove our lemma. 

\begin{scproof}[Proof of \cref{lemma:2-patch}]
We are given a hypergraph $H$ with $\rho(H)=:k$. That is, there is an $S\in\Stwo_n$ such that $P:=S-H$ satisfies $|P|= k$. We can assume without loss of generality that $S$ is the disjoint union of $P$ and $H$, since otherwise we may remove superfluous faces in $H$ by replacing it with $H\cap S$.
Our aim is show that there exists an $H$-patch with cost $O(k^{{3}/{4}}/n^{{1}/{12}})=o(k)$ with high probability. That is, a $P'$ such that $H\cup P'$ contains \emph{some} sphere in $\Stwo_n$ -- not necessarily the same as $S$.
Before detailing the formal proof, we present an overview. The proof proceeds in two steps: 
\begin{figure}
\centering
\subcaptionbox{A sphere with some faces marked red (if colour is shown), with a short planar separator $C$ (dark green) with a set ($Q$) of 8 interior vertices.}{\includegraphics[width=0.48\textwidth]{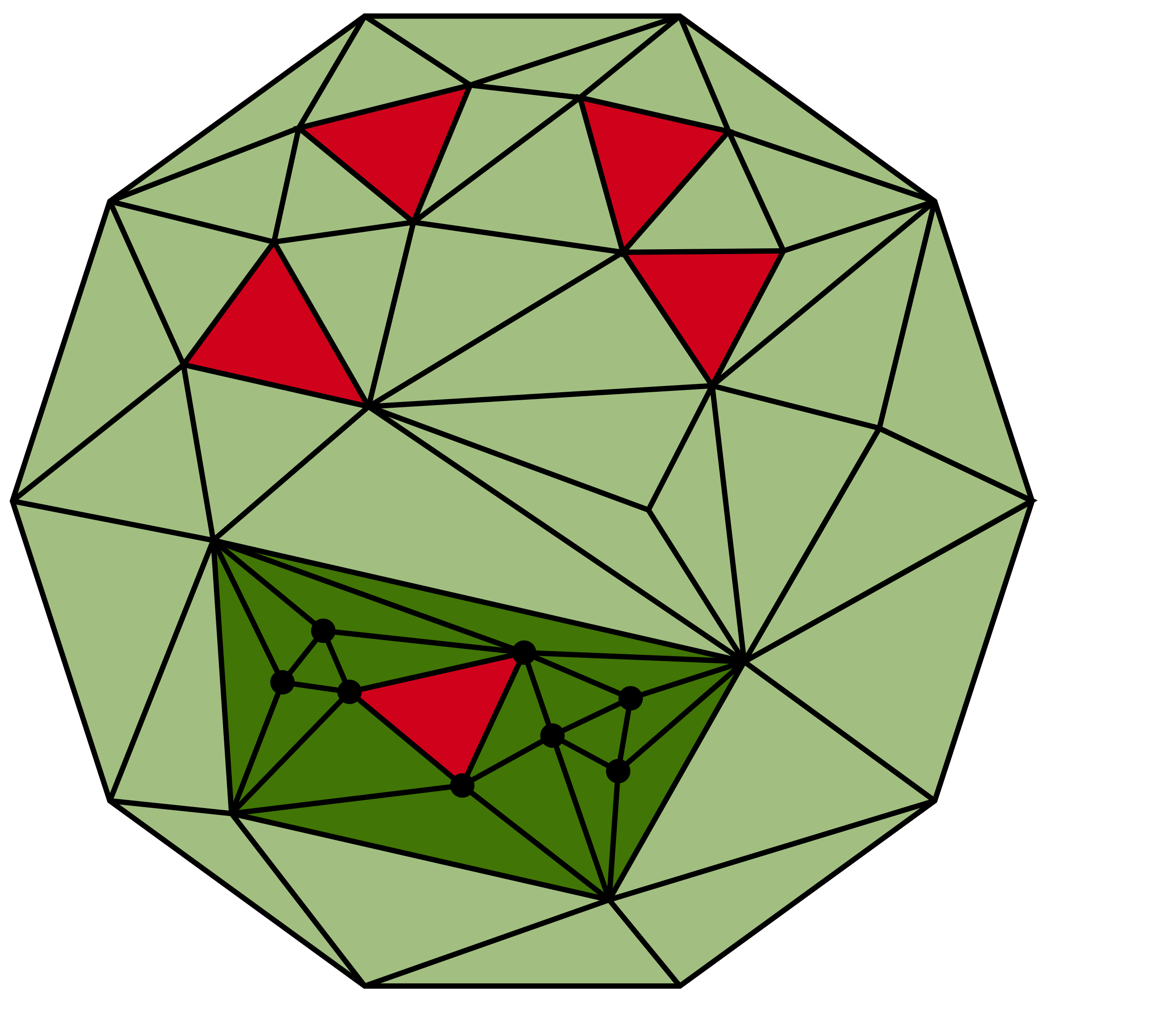}}%
\hfill
\subcaptionbox{Replace the interior of $C$ with a red triangulation without interior vertices, freeing up the vertices in $Q$. Match all red faces with free vertices.}{\includegraphics[width=0.48\textwidth]{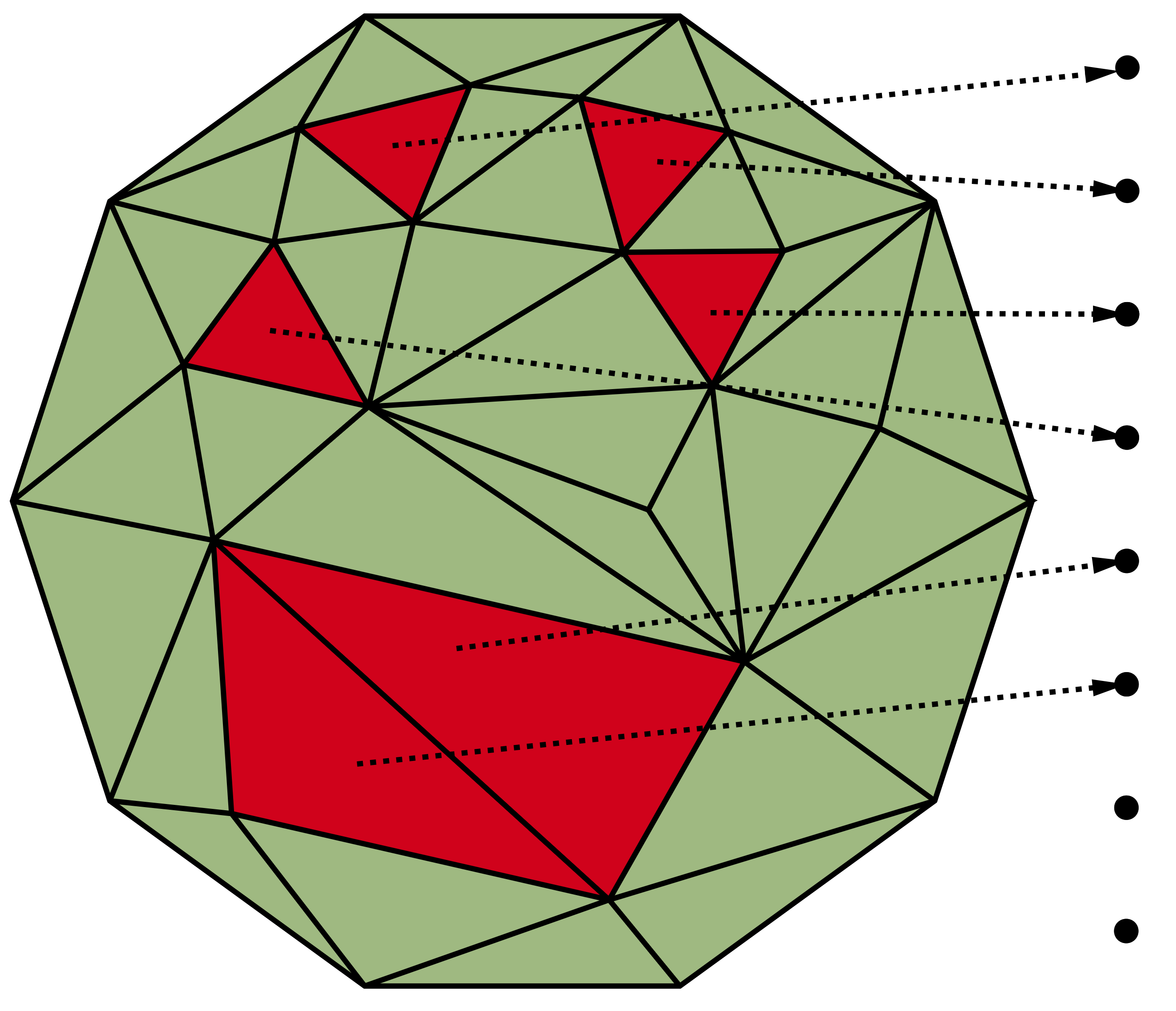}}%
\hfill
\subcaptionbox{Match the remaining free vertices with green faces.}{\includegraphics[width=0.48\textwidth]{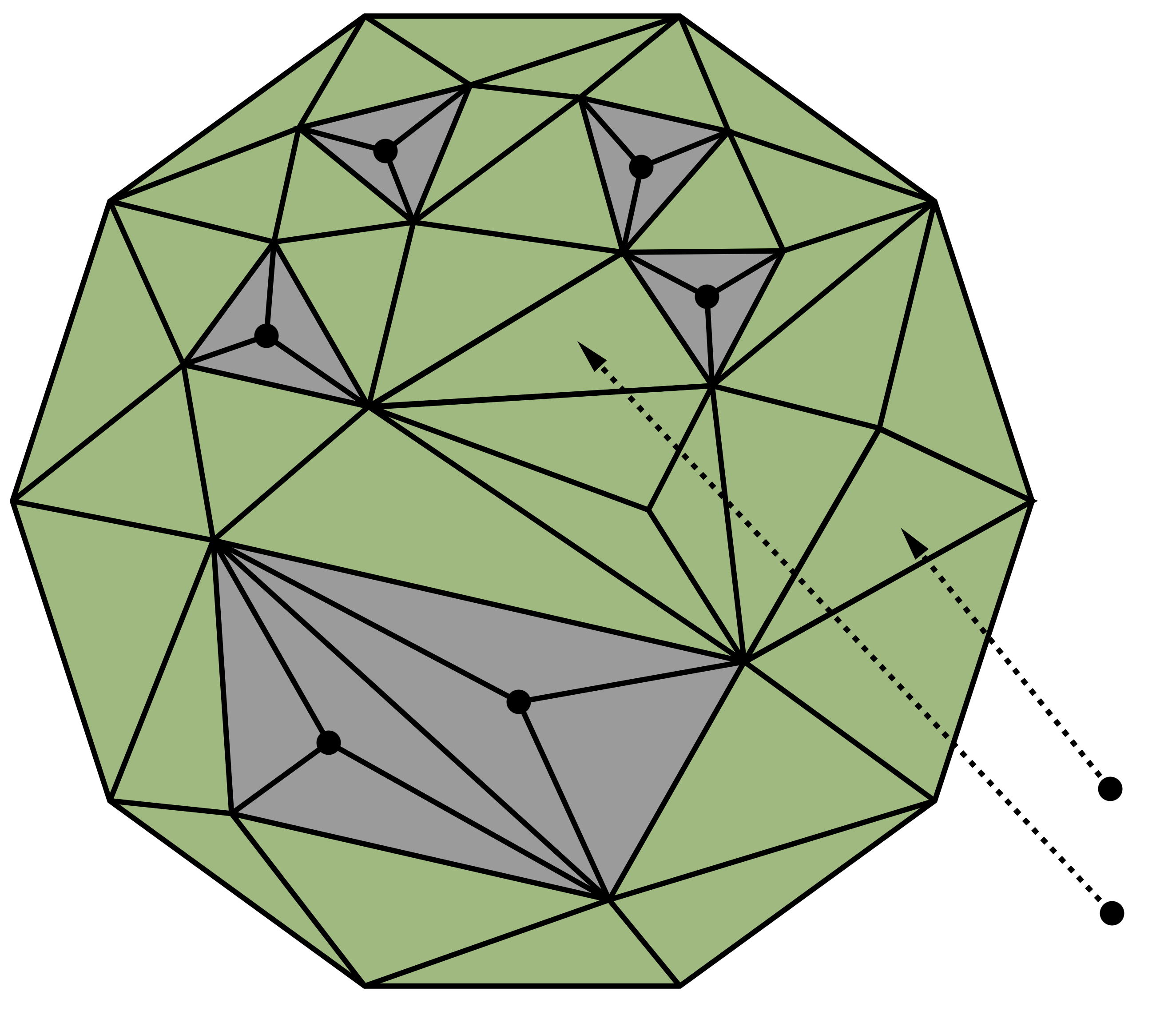}}%
\hfill
\subcaptionbox{\raggedright Patched sphere: Consists mostly of the original green faces, and some additional gray faces. }{\includegraphics[width=0.48\textwidth]{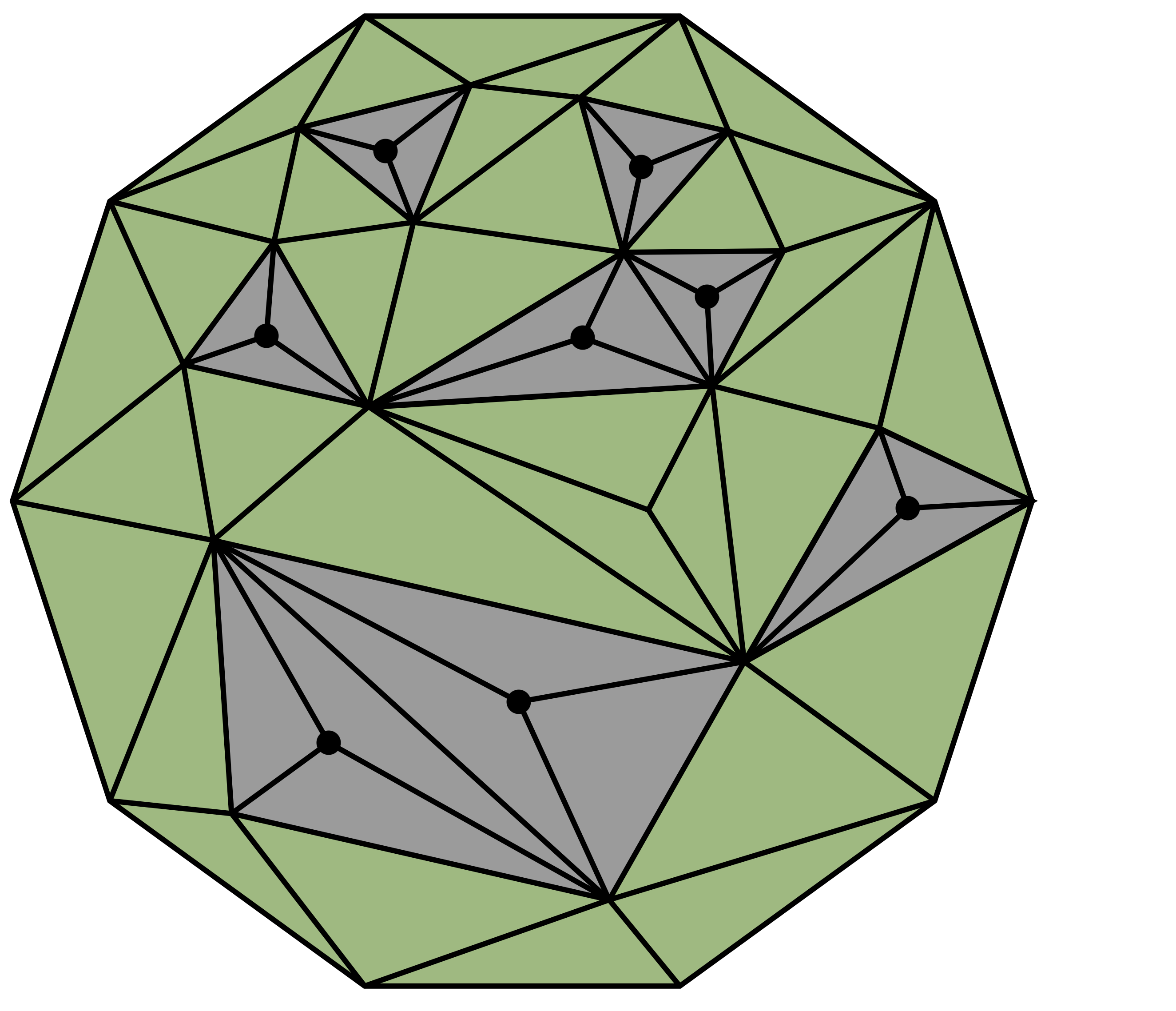}}%
\hfill
\caption{`Patching' a $2$-sphere. The aim is to remove the red faces (and possibly some green faces) and add some new faces of low total cost, such that the resulting complex is a $2$-sphere. We show that doing the matchings in (b) and (c) greedily gives sufficiently cheap faces in the barycentric subdivisions (gray).}
\label{fig:2patch}
\end{figure}
\begin{enumerate}
    \item We color the faces in $H$ green, and the faces in $P$ red. We use an asymmetric version of the Lipton--Tarjan planar separator theorem (\cref{lemma:planarsep} below)  to remove a set $Q$ of roughly $s$ vertices from $S$ (for some $s$ with $\sqrt{s}\ll k\ll s\ll n$), such that the boundary of $S\backslash Q$ is a cycle $C$ of length $O(\sqrt{s})$ (\cref{fig:2patch}(a)). Patch the hole in $S\backslash Q$ by arbitrarily triangulating $C$ without interior vertices. Color the newly added triangles red. This gives us a sphere $S'$ on $n-|Q|$ vertices, with at most $k+O(\sqrt{s})\approx k$ of its faces being red and the rest green, as well as $|Q|\approx s$ free vertices (\cref{fig:2patch}(b)).
    \item For any face $xyz$ in $S'$, we can use one of the free vertices $v\in Q$ to replace it with a barycentric subdivision, i.e.\ with the three faces $xyv$, $yzv$ and $zxv$. We do this for all red faces (\cref{fig:2patch}(b)), and for sufficiently many green faces to use up all vertices in $Q$ (\cref{fig:2patch}(c)).
    We want to minimize the total cost of the added faces of such a matching between faces in $S'$ and vertices in $Q$. We pick this matching greedily, and optimize the resulting cost over $s$.
\end{enumerate}

\subsubsection*{Step 1}
We proceed with the formal proof. For the first step, we will use the following asymmetric version of the Lipton--Tarjan planar separator theorem.\footnote{The condition that $n>2s$ is added by us to the lemma; the proof in \cite{planarsep} implicitly assumes that $s$ is not too large. In our application, $s\ll n$.}
\begin{lemma}[{\cite[Lemma 3.8 ]{planarsep}}]
\label{lemma:planarsep}
Let $s>0$ be a sufficiently large integer, and let $G$ be a triangulation of $\mathbb{S}^2$ with $n>2s$ vertices. Then $G$ contains a cycle $C$ of length at most $28\sqrt{s}$, such that $G\backslash C$ is disconnected into two sets of vertices, $Q$ and $Q'$, with the smaller of them satisfying $s/14 \leq |Q|\leq s$.
\end{lemma}
We apply this lemma with $G$ being the $1$-skeleton of $S$, and $s=s(k,n)$ some integer such that $\sqrt{s}\ll k\ll s\ll n$, whose precise value will be determined later.
Let $C$ and $Q$ be the resulting cycle and smaller set of vertices respectively, and let $q:=|Q|$, so that $s/14\leq q \leq s$ and $|C|\leq 28\sqrt{s}$. 

Let $D'$ be the simplicial complex obtained from $S$ by deleting 
the vertices in $Q$ (and all simplices containing any of them), and note that $D'$ is homeomorphic to a disc with boundary $C$.  

Let $D$ be an arbitrary triangulation of $C$ without  interior vertices. That is, $D$ is a $2$-manifold homeomorphic to a disc, with boundary $C$ and whose $1$-skeleton is an outerplanar graph. 
Easily, the number of triangular faces in $D$ equals the number of edges in $C$ minus 2.
Color the faces of $D$ red. Note that $D$ and $D'$ are both discs with the same boundary $C$, and so $S':=D \cup D'$ is a simplicial $2$-sphere on $n-q$ vertices. Let $R$ be the set of the red faces of $S'$, and note that $R\subseteq P\cup D^2$, where $D^2$ denotes the set of faces of $D$. Thus there are at most  $|P|+|D^2|\leq k+28\sqrt{s}=(1+o(1))k$ red faces of $S'$. Let $G$ be the set of green faces of $S'$, so that $S'$ is the disjoint union of $R$ and $G$.

\subsubsection*{Step 2}

For any face $\sigma=\vx\vy \vz $ in $S'$, we can replace $\sigma$  with a barycentric subdivision using any vertex $v\in Q$. That is, replace $\sigma$ with the faces ${\vx\vy v}, {\vy \vz  v}$ and ${\vz \vx v}$ at cost
\begin{equation*}
    X_{\sigma v}:= W_{\vx\vy v}+W_{\vy \vz  v}+W_{\vz \vx v}.
\end{equation*}
Doing so for all red faces will result in a $2$-sphere, possibly on fewer than $n$ vertices because not all $q$ vertices in $Q$ have been used. We can additionally replace some green faces with their barycentric subdivisions until all vertices in $Q$ have been used, finally obtaining a $2$-sphere on $n$ vertices.

For convenience, we construct an auxiliary graph: the complete bipartite graph $K$ with left vertex set $R\cup G$ and right vertex set $Q$, with edge costs $X_{\sigma v}$. Since $S'$ has $n-q$ vertices, the left set has $|R\cup G|=2(n-q)-4=:\ell$ vertices, while the right has $|Q|=q$. By construction, $\ell\gg q$.

Let $\mathcal{M}$ be the set of matchings in $K$ containing every vertex in $Q \cup R$
(and hence such that most vertices in $G$ are \emph{not}  matched).
Using the convention of identifying an edge $\sigma v$ of  $K$ with the set of simplices $\{{\vx\vy v},\, {\vy \vz  v}$,\, ${\vz \vx v}\}$, any $M\in \mathcal{M}$ is an $H$-patch. 
The barycentric subdivisions described above correspond to picking a matching in $\mathcal{M}$.  

The total cost of the matching $M$ is $X_M:=\sum_{\sigma v\in \mathcal{E}(M)} X_{\sigma v}$, and this is also the cost of the corresponding $H$-patch. We want to show that there exists an $M\in \mathcal{M}$ of low total cost, using a greedy exploration procedure.

However, this procedure would have been easier to analyse if the edge costs in $K$ were independent. In our case, though, the edge costs $X_{\sigma v}$ and $X_{\sigma' v}$ are not independent if the triangular faces $\sigma$ and $\sigma'$ overlap in an edge: If $\sigma=\vx\vy \vz $ and $\sigma'=\vx\vy \vz '$, then both $X_{\sigma v}$ and $X_{\sigma' v}$ depend on $W_{\vx\vy v}$.
Since this dependence is only between faces sharing an edge, consider the dual graph of the 1-skeleton of $S'$, where vertices corresponding to the triangular faces are connected if they share an edge. This graph is $3$-regular, 
thus it is $4$-colourable. Let $R_1,R_2,R_3$ and $R_4$ be the intersections of $R$ with the four colour classes.
Also let $G_1\subseteq G$ be the largest intersection of a colour class with $G$, so that $|G_1|\geq |G|/4=\Omega(n)$. At most $3|R|=o(n)$ faces in $G_1$ share an edge with a red face. Removing these faces leaves a set $G_0\subseteq G_1$, with $|G_0|=\Omega(n)$.
Crucially, the $R_i$'s and $G_0$ picked in this way are independent sets.
Finally, let $Q_1$, $Q_2$, $Q_3$ and $Q_4$ be a partition of $Q$ into four sets of similar sizes -- each with at least $|Q|/5$ vertices, say. We will first greedily match all vertices in each $R_i$ with vertices in the corresponding $Q_i$, and then greedily match the vertices in $Q$ that have not yet been matched, with vertices in $G_0$.
We will use the following Chernoff bound on the cost of these matchings. 
\begin{claim}[Chernoff bound]
\label{claim:Chernoff-sum-min-sum}
Let $U_1,U_2$ and $U_3$ be i.i.d.\ random variables following a $U(0,1)$-distribution. Assume $\{Z_{ij}\}_{1\leq i\leq k_0, \,1\leq j\leq n_0}$ are i.i.d.\ random variables such that $Z_{ij}\overset{d}{=} U_1+U_2+U_3$.
Then for any $b\geq 20k_0$,
\[
\Pr\left(\sum_{i}^{k_0}\min_{j} Z_{ij}\geq b n_0^{-1/3}\right)\leq e^{-b/25}
\]
\end{claim}
\begin{scproof}
    Observe that $\Pr(Z_{ij}<x+n_0^{-1/3}\mid Z_{ij}\geq x)$ is minimised when $x=0$, at $1/(6n_0)$. It follows that $\Pr(\min_jZ_{ij}<x+n_0^{-1/3}\mid\min_jZ_{ij}\geq x)\geq 1-\exp(-1/6)>0.15$.

Consequently $\lceil n_0^{1/3}\min_jZ_{ij}\rceil$ is dominated by a $\operatorname{Geo}(0.15)$ random variable, and these are independent for different $i$. Hence the probability that the sum is too large is at most the probability that a $\operatorname{Bin}(b,0.15)$ random variable takes a value less than $k_0\leq 0.05b$, which by a standard Chernoff bound is less than $\exp(-b/25)$.
\end{scproof}

Any of the $k_0=|R_i| \leq |R|< k+o(k)$ vertices in $R_i$ has at least $n_0=|Q_i|-k_0=\Omega(s)$ vertices in $Q_i$ to choose from. 
Apply \cref{claim:Chernoff-sum-min-sum} with $b= 25tk$ for some $t\geq 1$. Then with probability at least $1-e^{-tk}$, the greedy algorithm matches all vertices in $R_i$ with vertices in $Q_i$ with total cost at most $C_1k/s^{1/3}$ for some constant $C_1=C_1(t)$. 

After the vertices in $R$ have been matched with vertices in $Q$ in this way, $k_0=|Q|-|R|\leq s$ vertices in $Q$ remain unmatched, and each of them have at least $n_0=|G_0|-|Q|=\Omega(n)$ vertices in $G_0$ to choose between.
Apply \cref{claim:Chernoff-sum-min-sum}, now with $b= 25ts$. Then with probability at least $1-e^{-tk}$, the cost of this matching is at most $C_2s/n^{1/3}$ for some constant $C_2=C_2(t)$.

The total cost of the greedy matchings is then at most $4C_1k/s^{1/3}+C_2s/n^{1/3}$ with probability at least $1-5e^{-tk}$. Let
$s=k^{3/4}n^{1/4}$, making both terms above $\Theta(k^{3/4}/n^{1/12})$ (this is close to optimal).
The lemma follows by picking $t$ sufficiently large compared to $r$. 
\end{scproof}

\begin{remark}
In order to try to apply this proof method to spanning spheres in dimension $d\geq 3$, we would need something similar to the planar separator theorem (\cref{lemma:planarsep}) in higher dimension.
If we had such a theorem, then a weaker lower bound on $M(\Sd_n)$ would suffice to prove that this random variable is sharply concentrated. Such a separator theorem might also help approach Gromov’s aforementioned question~\cite{gromov-spaces-questions} with a divide-and-conquer strategy.

On the other hand, recall that the planar separator theorem was just used
in order to gain access to a sufficient number (at least $k$) of free vertices. We could instead have considered the complete hypergraph $K_n^{(d+1)}$ on $n$ vertices as a subgraph of the complete hypergraph $K_{n+k}^{(d+1)}$ on $n+k$ vertices, and used these additional $k$ vertices to perform the matching of step 2. Instead of inequalty (\ref{ineq:patch}), this would lead to an inequality relating $M(\Sd_{n+k})$ to the quantile functions of both $M(\Sd_{n})$ and $M(\Sd_{n+k})$.
Although heuristically $M(\Sd_{n+k})$ should only be slightly larger than $M(\Sd_{n})$, we neither know the size of their (typical) difference, nor even whether $\E M(\Sd_{n+k})$ is indeed larger than $\E M(\Sd_{n})$.
\end{remark}

\section{Concentration for \texorpdfstring{$d\geq 3$}{d at least 3}}
\subsection{Concentration assuming moderate growth}
\label{section:concentration}
We will prove the following theorem in the more general setting of an arbitrary family $\F$ of subhypergraphs; we denote the median of $M(\F)$ by $\mu$.
\begin{theorem}
\label{thm:conc-general}
Let $\F$ be a family of subhypergraphs of $K_n^{(d+1)}$ which for any $m\geq tn$ contains at most $K^m m^{\beta m}$ subgraphs with $m$ hyperedges, for some $t,K>0$ and $\beta<1/2$, and which contains no subgraph with fewer than $tn$ hyperedges.

Then $M(\F)$ is sharply concentrated: there exists $\delta>0$ such that \[\Pr(|M(\F)-\mu|>\mu^{1-\delta})\leq \exp(-n^{\delta})\] for $n$ sufficiently large.
\end{theorem}
Part (iii) of \cref{thm:dsphere} is a special case of \cref{thm:conc-general}.
To prove \cref{thm:conc-general}, we will need the following simple bound from \cite{patch-ineq}.
\begin{lemma}[Special case with $q=1$ of {\cite[Proposition 7.1]{patch-ineq}}]
\label{lemma:uppertail}For any $t\geq 0$,
\[
    \Pr(M(\F)>t\mu )\leq 2^{1-t}.
\]
\end{lemma}

\begin{scproof}[Proof of \cref{thm:conc-general}]
We will show this for the explicit value $\delta=(1/2-\beta)/4>0$.
Split $\F$ into two families: $\F'$ consisting of those sets $F\in \F$ of size at most some $m_1\gg n$ (to be determined later), and $\F'':=\F-\F'$ consisting of the larger $F$'s.
Note that $M(\F)=\min(M(\F'),M(\F''))$. 
We will show that $M(\F)\leq \mu^{1+4\delta}\ll M(\F'')$ w.h.p., from which it follows that $M(\F)=M(\F')$ w.h.p. We will then show that $M(\F')$ (and hence $M(\F)$) is sharply concentrated.

Applying the first moment bound \cref{LB:general-graph-family} to $\F$ gives $\mu=\Omega((nt)^{1-\beta})=\Omega(n^{1/2+4\delta})$.
Similarly, applying \cref{LB:general-graph-family} to $\F''$ gives $M(\F'')=\Omega(m_1^{1/2+4\delta})$.

The bound we will prove on the fluctuations of $M(\F')$ around its median will be roughly $\sqrt{m_1}$. We therefore set $m_1:=\mu^{2-4\delta}$. Observe that, since $0<4\delta<1/2$, we have $(2-4\delta)(1/2+4\delta)>1+4\delta$.

Using the first moment bound on $M(\F'')$ above, we obtain \[M(\F'')=\Omega(\mu^{(2-4\delta)(1/2+4\delta)})\gg \mu^{1+4\delta},\] and similarly 
$m_1=\Omega(n^{(1/2+4\delta)(2-4\delta)})\gg n^{1+4\delta}$.

To upper bound $M(\F)$, we use \cref{lemma:uppertail} with $t=\mu^{4\delta}\gg 1$, giving 
\[
\Pr(M(\F)\geq \mu^{1+4\delta})\leq 2^{1-\mu^{4\delta}}=o(1).
\]

We now focus on $M(\F')$, and show that it is sharply concentrated around its median. $\F'$ is trivially $(k,k,\eps)$-patchable for any $k$ and $\eps$: Simply put back the $k$ elements that were removed, at cost at most $k$.\footnote{For this trivial patchability condition, \Cref{thm:patchability-ineq} reduces to the Talagrand inequality, in particular the `certifiability' corollary as found in~\cite{probmeth}. It would have been slightly easier to apply this inequality directly instead of using \Cref{thm:patchability-ineq}, however we use the latter to avoid stating the inequality and definitions from~\cite{probmeth}.}
We apply \cref{thm:patchability-ineq} for some $\eps=\eps(n)$ tending to zero not too fast ($\log(\eps^{-1})\leq \mu^{2\delta}$, say), and $\lambda=k:=\sqrt{8\log(\eps^{-1})m_1}$. 
We then have $\log(\eps^{-1})m_1\leq \mu^{2-2\delta}$ and hence $\lambda\leq \mu^{1-\delta}$.
Letting $\mu'$ be the median of $M(\F')$, \cref{thm:patchability-ineq} then gives us that
\[
\Pr\left(\big|M(\F')-\mu'\big|\geq 3\mu^{1-\delta}\right) \leq 3\eps.
\]
Since $M(\F)=M(\F')$ w.h.p., this inequality also holds for $M(\F)$ if we add an extra $o(1)$ term to the right-hand side. 
The right-hand side is strictly less than $1/2$, so we must furthermore have  $|\mu'-\mu|\leq 3\mu^{1-\delta}$. By the triangle inequality, $\big|M(\F)-\mu\big|\leq 6\mu^{1-\delta}$ with high probability.
\end{scproof}

\subsection{Small subfamilies and locally constructible spheres} \label{section:lc-spheres}
In this section we prove \cref{LC-sphere-scaling}, but before doing so let us discuss an important family of spheres to which it applies, namely the \emph{locally constructible spheres}. (The reader may skip this discussion and proceed directly to the proof.)
\medskip

Tutte's proof of \cref{tutte-enum} involves a clever and intricate use of generating functions. It is however fairly easy to see that $B_{n,2n-4}^{(2)}\leq C^n$ for \emph{some} constant $C$, and we include a sketch of that argument here.

\begin{scproof}[Proof sketch]
Any $2$-sphere can be constructed in a two-stage process. First, glue together $2n-4$ triangular faces along edges in such a way that they form a disc without interior vertices. That is, the face-dual graph is a tree. This forms a triangulation of a $(2n-2)$-gon, and the number of such triangulations is given by the Catalan number $C_{2n-4}=o(4^n)$. Next, pair up the edges along the boundary with a planar perfect matching, and glue matched pairs of edges together. The number of planar matchings can also be counted by the Catalan numbers, so again grow exponentially with $n$. We can obtain any given isomorphism class of $2$-spheres (in many different ways, as we have counted them), since we may base the first stage on any spanning tree of the dual graph. 
\end{scproof}

In dimension $d=3$, the argument above fails in the second stage. After constructing a $3$-ball, we are no longer constrained to \emph{planar} matchings when pairing up faces on its boundary.

The way the edges of the disc are glued together in dimension $d=2$ can also be thought of as a sequence of `local' glueings: First glue together two adjacent edges. This shrinks the length of the boundary by two, and causes two previously non-neighbouring edges to become neighbours. Repeat this process until the boundary is eliminated.

Generalizing this notion of local glueings to higher dimensions, 
Durhuus \& J\'onsson \cite{LC-spheres-definition} introduced the class of \emph{locally constructible} manifolds (see \cref{def:lc-phere} below).
They showed that there are only exponentially many locally constructible manifolds with $m$ facets. As sketched above, all $2$-spheres are locally constructible. Durhuus \& J\'onsson posed the question of whether this is also true for $3$-spheres, but it was answered in the negative by Benedetti \& Ziegler~\cite{LC-spheres-characterization}.
For more background on locally constructible spheres, as well as how they relate to shellable and collapsible simplicial complexes, we refer the reader to \cite{LC-spheres-characterization}.

\begin{definition}
\label{def:lc-phere}
The \textit{locally constructible} (LC) $d$-balls are those simplicial $d$-balls that can be constructed by a sequence of the following two moves, starting from a single $(d+1)$-simplex.
\begin{enumerate}
\item Glue a new $(d+1)$-simplex along a $d$-simplex to a $d$-simplex on the boundary.
\item Glue two $d$-simplices together, provided that they both lie on the boundary and already overlap in a $(d-1)$-simplex.
\end{enumerate}
A \textit{locally constructible sphere} is the boundary of a locally constructible ball.
\end{definition}

We now show that $S^*_{n,d}$ is a locally constructible sphere, and so \cref{LC-sphere-scaling} applies to the locally constructible spheres.
\begin{scproof}[Proof of \cref{its-a-sphere}]
Let $\mathrm{Cone}=\mathrm{Cone}_z$ be the cone operator that takes a simplicial complex $\mathcal{K}$ and adds to it a new vertex $z$ and a simplex $K \cup \{z\}$ for every $K\in \mathcal{K}$.

Note that $B(C)=\mathrm{Cone}^d(C)=\mathrm{Cone}_{v_1}(\mathrm{Cone}_{v_2}(\cdots \mathrm{Cone}_{v_d}(C)\ldots ))$.
$\mathrm{Cone}(C)$ is homeomorphic to the $2$-ball (that is, the disc), and is clearly LC. Note that the cone over a $k$-ball is homeomorphic to the $(k+1)$-ball, and furthermore that the cone of an LC $k$-ball is also LC (by considering coning every step of the construction). Hence $\mathrm{Cone}^d(C)$ is homeomorphic to the $(d+1)$-ball, its boundary is homeomorphic to the $d$-sphere, and both are LC. 

If $y,y'$ are the two neighbours of some vertex $x$ in $C$, then $\{x,v_1,\ldots,v_d\}$ lies in both $\partial f_{xy}$ and $\partial f_{xy'}$, and hence they cancel out in $S_C=\partial B(C)$.
On the other hand, any $d$-simplex containing both $x$ and $y$ (with $xy\in E(C)$) occurs only once, namely in $\partial f_{xy}$. Hence the facets are
\[
\{f_{xy}-\{v_i\}: xy\in E(C), 1\leq i \leq d\},
\]
and there are $|E(C)|\cdot d = (n-d)\cdot d$ such facets.
\end{scproof}

Benedetti \& Pavelka \cite{BP21} extended \cref{def:lc-phere} to give the class of $t$-LC spheres, where the requirement to overlap in a $(d-1)$-simplex is relaxed to a $(d-t)$-simplex. They showed that the larger class of $2$-LC spheres is also exponentially bounded, and thus \cref{LC-sphere-scaling} also applies to this class. However, even for $d=3$ (where these are also known as Mogami spheres), not all spheres are $2$-LC \cite{Ben17}.

\begin{scproof}[Proof of \cref{LC-sphere-scaling}]
The argument is similar to that of the proof of \cref{thm:conc-general}. First, split $\mathcal{P}$ into two families: $\mathcal{P}'$, consisting of all spheres in $\mathcal{P}$ with at most $m_1:=(d+\eps)n$ facets, and $\mathcal{P}'':=\mathcal{P}-\mathcal{P}'$.
For the upper bound, the sphere $S_C$ in the proof of \cref{UB:polar-sphere} is in $\mathcal{P}$ by assumption.
Furthermore, it has $d(n-d)<m_1$ facets, and hence it also lies in $\mathcal{P}'$. Thus $M(\mathcal{P})\leq M(\mathcal{P}')\leq Cn^{1-\frac{1}{d}}$ with high probability, for some $C$. 

Next, there are at most $2^{O(m)} n!$ spheres in $\mathcal{P}$ with $m$ facets, also by assumption.
For $m$ in the range $dn\leq m\leq m_0$, we have  $n!\leq m^{m/d}$, while for $m\geq m_0$, we have $n!\leq m^{m/(d+\eps)}$.
Using \cref{LB:general-graph-family} on  $\mathcal{P}'$ and $\mathcal{P}''$ with $\beta=1/d$ and $\beta=1/(d+\eps)$ respectively therefore yields $M(\mathcal{P}')\geq c'n^{1-\frac{1}{d}}$ and $M(\mathcal{P}'')\geq c''n^{1-\frac{1}{d+\eps}}$, with high probability, for some $c',c''>0$. Together with the upper bound on $M(\mathcal{P}')$ above, we deduce $M(\mathcal{P}'')\gg M(\mathcal{P}')$ w.h.p., and hence $M(\mathcal{P})=M(\mathcal{P}')$ w.h.p.

The proof of concentration then proceeds just as in \cref{thm:conc-general} by using the Talagrand inequality to show that $M(\mathcal{P}')$ is sharply concentrated.
\end{scproof}

\medskip
The aforementioned question of Durhuus \& J\'onsson would have implied the following, if the answer had been positive, using the method of \cref{sec:bounds} or \ref{section:lowerbound}.

\begin{conj}
\label{conj:3sphere} There is a constant $c>0$ such that for all sufficiently large $n$,
\[
\mu^{(3)}_n \geq  cn^{2/3}.\]
\end{conj}

Here $\mu^{(3)}_n$ denotes the median of $M(\Sthree)$ as in \cref{thm:3sphere}.

\section{Outlook} \label{Outlook}

We conjecture that the upper bound on 
$\mu^{(d)}_n$ provided by \cref{thm:dsphere} (i) is sharp: 
\begin{conj}
\label{conj c} 
For every $d\geq 2$ there is a constant $c_d>0$ such that 
\[
\mu^{(d)}_n = c_d n^{1-1/d} + o(n^{1-1/d}).\]
\end{conj}
We remark that this is open even for $d=2$, in which case \cref{thm:2sphere} bounds the possible values of $c_2$ within the interval $[\alpha, \frac{e}{2}\alpha]$.

\medskip
As mentioned in the introduction, we can consider a more general model where the cost of a triangulation in $\Sd_n$ influences its probability of being selected via a Boltzmann distribution. To make this precise, let $\lambda=\lambda_{d,n}$ denote the Lebesque measure on the cube $Q:= [0,1]^{\binom{n}{d+1}}$, to be thought of as a random weighing of $K_n^{(d)}$, assigning independent $U(0,1)$-distributed costs to its hyperedges. For each triangulation $S\in \Sd_n$, and each weighing $w\in Q$,  define the \emph{energy} $H(S,w):= \sum_{e\in S} w(e)$, i.e.\ the total cost of $S$ with respect to $w$. Finally, choose a random sphere $\Rbn{d} \in \Sd_n$ with probability density 
\[f_\beta(S):= \frac{e^{-\beta H(S,w)}}{Z_\beta},\]
where $Z_\beta$ denotes the \emph{partition function} $\sum_{S\in \Sd_n} \int_Q  e^{-\beta H(S,w)} d\lambda(w)$, and $\beta\in [0,\infty)$ is the \emph{inverse temperature} parameter. Thus the probability that $\Rbn{d}$ coincides with a given $S\in \Sd_n$ is $\frac{1}{Z_{\beta}}\int_Q  e^{-\beta H(S,w)} d\lambda(w)$. 

All results and questions of this paper can be studied with $M(\Sd_n)$ replaced by \Rbn{d}. In particular, we can let $\mu^{(d)}_{n,\beta}$ be the median of \Rbn{d}, and adapt \Cref{conj c} as follows:
\begin{conj}
\label{conj c beta} 
For every $d\geq 2$ and $\beta\in [0,\infty)$, there are constants $e^{(d)}(\beta), c_{d,\beta}>0$ such that 
\[
\mu^{(d)}_{n,\beta} = c_{d,\beta} n^{e^{(d)}(\beta)} + o(n^{e^{(d)}(\beta)}).\]
\end{conj}

Note that for $\beta=0$ we  recover 
the uniform $n$-vertex triangulation of the $d$-sphere. In particular, for $d=2$, we have $\mu^{(2)}_{n,0}=\Theta(n)$. We believe that $e^{(2)}(\beta)=1$ for every $\beta\in [0,\infty)$. On the other hand, our \cref{thm:2sphere} says that $\mu^{(2)}_{n,\infty}=\Theta(n^{1/2})$ if we interpret $\beta=\infty$ appropriately. 
We know nothing about $e^{(d)}(\beta)$ in any dimension $d\geq 3$.

\begin{problem}
    For $d=1,2,\ldots$, determine  $e^{(d)}(\beta)$. In particular, is $e^{(d)}(\beta)$ a constant function of $\beta\in [0,1)$ for every $d$? 
\end{problem}

If \Cref{conj c beta} is true, it would be interesting to study $c_{d,\beta}$ as a function of $\beta$, in particular, to decide its smoothness. 

\medskip

Recall that the uniform $n$-vertex triangulation \Rntwo\ of the $2$-sphere converges, in the sense of Benjamini \& Schramm \cite{BeSchrRec}, as proved by Angel \& Schramm \cite{AnSchrUn}. It would be very interesting to generalise this to other values of $\beta$ (and dimensions), and to understand how the limit object varies with $\beta$: 

\begin{conj}
    For every $d\geq 2$, and every $\beta\in [0,\infty]$, \Rbn{d}\  converges in the sense of Benjamini \& Schramm as $n\to \infty$.
\end{conj}
There are various ways to formulate this conjecture precisely; in its simplest form, we ask about the convergence of the $1$-skeleton of \Rbn{d}, considered as a graph with the costs ignored. 

\medskip
There is another well-known convergence result about the uniform $n$-vertex triangulation of the $2$-sphere: if we endow \Rntwo\  with the graph-metric, and rescale it appropriately (in such a way that its diameter stays bounded), then this metric space converges in distribution in the Gromov--Hausdorff sense as $n\to \infty$. The limit object is a random metric space, called the \emph{Brownian map}, which  is homeomorphic to $\mathbb{S}^2$ and yet has Hausdorff dimension $4$ \cite{LeGBro,MieBro,BaezBro}. The Brownian map and similar models of random fractal geometry are being studied extensively due to connections with mathematical physics, see e.g.\ \cite{GwyRan} and references therein. Our next question is how this generalises to our setup. Let $|\Rbn{d}|$ denote the metric space obtained from \Rbn{d}\ by rescaling the metric by a factor of $1/{\operatorname{diam}(\Rbn{d})}$.
\begin{problem}
    Does $|\Rbn{d}|$ converge in distribution in the Gromov--Hausdorff sense for every $d\geq 2$, and every $\beta\in [0,\infty]$? If yes, what is the Hausdorff dimension of the limit space as a function of $\beta$? In particular, is this function monotone in $\beta$?
\end{problem}

\medskip
All of the  definitions and questions of this section can be reformulated, and are equally interesting, when the random costs are placed on the $1$-cells rather than the $d$-cells of $K_n^{d}$, as in \cref{cor}. 

\section*{Acknowledgement}

We are grateful to Christoforos Panagiotis for a helpful discussion about \Cref{Outlook}.

\bibliographystyle{plain}
\bibliography{bib}
\end{document}